\documentclass[a4paper,12pt]{article}
\usepackage[margin=1in]{geometry}  

\usepackage{graphicx}              
\usepackage{amsmath}
\usepackage{amscd}               
\usepackage{amsfonts} 
\usepackage{amssymb}             
\usepackage{amsthm}                

\newtheorem{thm}{Theorem}[section]
\newtheorem{defn}[thm]{Definition}
\newtheorem{lem}[thm]{Lemma}
\newtheorem{prop}[thm]{Proposition}
\newtheorem{cor}[thm]{Corollary}

\newtheorem{rmk}[thm]{Remark}

\newcommand{\RR}{\mathbb{R}}      
\newcommand{\NN}{\mathbb{N}}

\newcommand{\oO}{\mathcal{O}}

\newcommand{\eE}{\mathcal{E}}
\newcommand{\kK}{\mathcal{K}}
\newcommand{\sS}{\mathcal{S}}
\newcommand{\wW}{\mathcal{W}}
\newcommand{\pP}{\mathcal{P}}
\newcommand{\xX}{\mathcal{X}}

\newcommand{\EE}{\mathbb{E}}     

\newcommand{\lfl}{\left\lfloor }  
\newcommand{\rfl}{\right\rfloor} 

\begin{document}

\title{Concentration and exact convergence rates for expected Brownian signatures}

\author{\textsc{Hao Ni and Weijun Xu}\\
\textit{University of Oxford}}

\maketitle

\abstract{The signature of a $d$-dimensional Brownian motion is a sequence of iterated Stratonovich integrals along the Brownian paths, an object taking values in the tensor algebra over $\RR^{d}$. In this note, we derive the exact rate of convergence for the expected signatures of piecewise linear approximations to Brownian motion. The computation is based on the identification of the set of words whose coefficients are of the leading order, and the convergence is concentrated on this subset of words. Moreover, under the choice of projective tensor norm, we give the explicit value of the leading term constant. }

\bigskip

\section{Introduction}

Let $(e_{1}, \cdots, e_{d})$ be the standard basis of $\RR^{d}$, $d \geq 2$, and let
\begin{align*}
B_{t} = \sum_{j=1}^{d} B_{t}^{j}e_{j}, 
\end{align*}
where $B_{t}^{j}$'s are independent standard one dimensional Brownian motions. The signature of $B$ is a sequence of Stratonovich iterated integrals along the sample paths (\cite{Hambly and Lyons}, \cite{LeJan and Qian}). We give a formal definition below.

\bigskip

\begin{defn} \label{definition of signatures}
For every $n \geq 1$ and every word $w = e_{i_{1}} \cdots e_{i_{n}}$ with length $n$, define
\begin{align} \label{Stratonovich signatures}
C^{w}_{s,t} = \int_{s < u_{1} < \cdots < u_{n} < t} \circ dB_{u_{1}}^{i_{1}} \cdots \circ dB_{u_{n}}^{i_{n}}
\end{align}
in the sense of Stratonovich integral. For each $n \geq 0$, let $X_{s,t}^{n}(B) = \sum_{|w|=n}C^{w}_{s,t}$, where the sum is taken over all words of length $n$. We use the convention $C^{w}_{s,t} \equiv 1$ if $w$ is the empty word. Then, the series
\begin{align*}
X_{s,t}(B) = \sum_{n=0}^{+\infty}X_{s,t}^{n}(B)
\end{align*}
is the (Stratonovich) signature of B over time interval $[s,t]$. 
\end{defn}

\bigskip

\begin{rmk}
It is sometimes more convenient to write the signatures in terms of tensors, i.e., 
\begin{align*}
X_{s,t}^{n} = \int_{s < u_{1} < \cdots u_{n} < t} \circ dB_{u_{1}} \otimes \cdots \otimes \circ dB_{u_{n}}, 
\end{align*}
and $C^{w}_{s,t}$ defined in \eqref{Stratonovich signatures} is the coefficient of $w$ in $X$. This is equivalent to Definition \ref{definition of signatures}. 
\end{rmk}

\bigskip

The study of the signature of a path dates back to K.T.-Chen in 1950's. In a series of papers (\cite{Chen 1957}, \cite{Chen 1958}, \cite{Chen 1977}), he developed algebraic properties of these multiple iterated integrals, and showed that piecewise smooth paths are characterized by their iterated path integrals over a fixed time interval. Hambly and Lyons (\cite{Hambly and Lyons}) gave a quantitative version of this result, and extended it to all paths of bounded variation. They showed that, paths of bounded variation in $\RR^{d}$ are uniquely determined by their signatures up to tree-like equivalence. 

Lyons (\cite{Lyons 1998}) studied the signatures of paths that are not necessarily piecewise smooth. He realized that the key properties in defining an integration theory along non-regular paths is this sequence of iterated integrals rather than the path itself alone. This idea led to the development of rough path theory. 

As for random paths, the expected signature is an important object to study as it determines the law of compactly supported measure on path space, and this is anticipated to be true for more general stochastic processes, the foremost example being Brownian motion. The computation of the expected signature of Brownian motion also leads to cubature on Wiener space (\cite{Lyons and Victoir}). 

The expected signature for Brownian motion was first derived by Fawcett (\cite{Fawcett}), and then independently by Lyons and Victoir (\cite{Lyons and Victoir}). In this note, we show that the expected signature of piecewise linear approximation to Brownian motion with mesh size $\frac{1}{M}$ converges to that of Brownian motion with rate $\frac{1}{M}$. This rate can be used to estimate the efficiency in some cubature algorithms. Moreover, under the choice of projective tensor norm, we give the explicit value of the leading term constant. This is an example where the projective tensor norm is more useful than the usual Hilbert Schmidt norm (see Theorem 9 in \cite{Hambly and Lyons} for another example). 

More precisely, let $B^{(M)}$ denote the piecewise linear approximation to Brownian motion with mesh size $\frac{1}{M}$. Let
\begin{align*}
\phi(T) = \EE X_{0,T}(B), \qquad \phi^{M}(T) = \EE X_{0,T}(B^{(M)}), 
\end{align*}
then our main theorem is the following.

\bigskip

\begin{thm} \label{main theorem}
For each $n \geq 0$, let $\pi_{n}$ denote the projection from the tensor algebra to $(\RR^{d})^{\otimes n}$. Then, 

\begin{flushleft}
(i) $\pi_{2}(\phi(T)) = \pi_{2}(\phi^{M}(T))$, and $\pi_{2n-1}(\phi(T)) = \pi_{2n-1}(\phi^{M}(T)) = 0$ for all $n \geq 1$. 
\end{flushleft}

\begin{flushleft}
(ii) For each $n \geq 2$, if $\RR^{d}$ is endowed with the $l_{1}$ norm, and $(\RR^{d})^{\otimes 2n}$ is given the projective tensor norm (to be defined in the next section), then
\end{flushleft}
\begin{align} \label{exact rate}
\lim_{M \rightarrow +\infty} \frac{M}{T} \left\| \pi_{2n}(\phi(T)) - \pi_{2n}(\phi^{M}(T)) \right\| = \frac{d-1}{3 \cdot (n-2)!}  \bigg(\frac{dT}{2} \bigg)^{n-1}. 
\end{align}
\end{thm}

\bigskip

The first part of the theorem is an immediate consequence of the basic properties of $\phi(T)$ and $\phi^{M}(T)$, which we will establish in section 3 below. The proof of the second claim is more involved. The core part of the proof is to identify for each $n$ the words whose coefficients are of order $\frac{1}{M}$, which turns out to be a rather small subset of words of length $2n$. The coefficients of all other words are of order $\oO(\frac{1}{M^{2}})$. That is to say, $\left\| \pi_{2n}(\phi(T)) - \pi_{2n}(\phi^{M}(T)) \right\|$ is concentrated on this small subset. We will give precise meaning in section 4 below. 

It should be noted that the exact value of the right hand side of \eqref{exact rate} depends on the choice of tensor norm and the equal space piecewise linear approximation. However, the concentration decribed above is due to the intrinsic nature of Brownian signatures, and remains unchanged under different tensor norms. 

Our paper is organized as follows. In section 2, we introduce the notion of tensors and the projective tensor norm. In section 3, we give some formulae and basic properties of the expected signatures of Brownian motion and its piecewise linear approximations. Section 4 is devoted to the proof of the main theorem.

\bigskip

\textsc{Acknowledgement.} We wish to thank our supervisor Terry Lyons for his support and helpful discussions.

\bigskip

\section{The projective tensor norm}

For each $n \geq 1$, the $n$-tensor space $(\RR^{d})^{\otimes n}$ is a real vector space with basis
\begin{align*}
\{e_{i_{1}} \cdots e_{i_{n}}: 1 \leq i_{1}, \cdots, i_{n} \leq d\}. 
\end{align*}
The tensor algebra over $\RR^{d}$ is defined by the direct sum
\begin{align*}
T(\RR^{d}):= \RR \oplus \RR^{d} \oplus \cdots \oplus (\RR^{d})^{\otimes n} \oplus \cdots. 
\end{align*}
Although it is common to identify $(\RR^{d})^{\otimes n}$ with $\RR^{d^{n}}$, which gives the Hilbert Schmidt norm, in many cases, the projective norm is more significant and useful. We give the definition below.

\bigskip

\begin{defn}
The projective tensor norm on $(\RR^{d})^{\otimes n}$ is defined by
\begin{align*}
\left\| v  \right\|:= \inf \bigg\{ \sum_{i} \left\| v_{1,i} \right\| \cdots \left\| v_{n,i} \right\|:  v = \sum_{i} v_{1,i} \otimes \cdots \otimes v_{n,i} \bigg\}. 
\end{align*}
\end{defn}

One should note that the projective tensor norm may vary according to different norms on $\RR^{d}$. In this paper, we choose $l_{1}$ norm on $\RR^{d}$. It is easy to deduce from the definition that if $x \in (\RR^{d})^{\otimes n}$ can be expressed as $x = \sum_{|w|=n} C^{w}w$, then
\begin{align*}
\left\| x \right\| = \sum_{|w|=n} |C^{w}|. 
\end{align*}

\bigskip

\textsc{Notations.} In the rest of the paper, $\left\|  \cdot \right\|_{n}$ will denote the projective tensor norm on $(\RR^{d})^{\otimes n}$. We will omit the subscript $n$ and simply write $\left\|  \cdot \right\|$ if no confusion may arise. We use $\pi_{n}$ to denote the projection from $T(\RR^{d})$ onto $(\RR^{d})^{\otimes n}$. Also, if $x \in T(\RR^{d})$, and $w$ is a word, then $C^{w}(x)$ will denote the coefficient of $w$ in $x$. Finally, for fixed $T$ and $M$, we write $\Delta t = \frac{T}{M}$.

\bigskip

\section{The expected signatures of Brownian motion and its piecewise linear approximations}

In this part, we give some formulae and propositions of $\phi(T)$ and $\phi^{M}(T)$. We first introduce some notations. For any word $w$, let $N_{i}(w)$ denote the number of occurences of the letter $e_{i}$ in $w$. For each $n \geq 0$, let
\begin{align*}
\sS_{2n} = \{w: w = e_{i_{1}}^{2} \cdots e_{i_{n}}^{2}, 1 \leq i_{1}, \cdots, i_{n} \leq d\}, 
\end{align*}
and
\begin{align*}
\kK_{2n} = \{w: |w| = 2n, \phantom{1} N_{i}(w) \text{is even for all $i$}\}. 
\end{align*}
The following formula for $\phi(T)$ was proven by Fawcett in \cite{Fawcett}.

\bigskip

\begin{prop} \label{expected signature for Brownian motion}
Let $B$ be a $d$-dimensional Brownian motion. Then, 
\begin{align*}
\phi(T) = \EE [X_{0,T}(B)] = \exp \bigg[ \frac{T}{2} \sum_{j=1}^{d} e_{j} \otimes e_{j} \bigg]. 
\end{align*}
\end{prop}

\bigskip

It is immediate from the proposition that if $w \in \sS_{2n}$ for some $n$, then
\begin{align*}
C^{w}(\phi(T)) = \frac{1}{n!} \bigg( \frac{T}{2}  \bigg)^{n}, 
\end{align*}
and $C^{w}(\phi(T)) = 0$ for all other $w$'s.

\bigskip

\begin{lem} \label{expected signature for piecewise linear approximations}
Fix an arbitrary $n \in \NN$. If $w \in \kK_{2n}$ such that $N_{k}(w) = 2 i_{k}$ for $k = 1, \cdots, d$, then for each $t \geq 0$, we have
\begin{align*}
C^{w}(\phi^{1}(t)) = \frac{\lambda_{w}}{n!} \bigg( \frac{t}{2} \bigg)^{n}, 
\end{align*}
where $\lambda_{w} = \begin{pmatrix} n\\ i_{1}, \cdots ,i_{d} \end{pmatrix} \big/ \begin{pmatrix} 2n \\ 2i_{1}, \cdots, 2i_{d} \end{pmatrix} \leq 1$. On the other hand, $C^{w}(\phi^{1}(t)) = 0$ for all $w \in \kK_{2n-1}$ and all $t \geq 0$. 
\end{lem}
\begin{proof}
If $\gamma = (\gamma^{1}, \cdots, \gamma^{d})$ is a straightline, and $w = e_{j_{1}} \cdots e_{j_{k}}$, then
\begin{align*}
C^{w}(X_{0,t}(\gamma)) = \frac{1}{k!} \gamma^{j_{1}}(t) \cdots \gamma^{j_{k}}(t). 
\end{align*}
Taking expectation of both sides gives
\begin{align*}
C^{w}(\phi^{1}(t)) = \frac{1}{k!} (\EE (B^{1}_{t})^{i_{1}}) \cdots (\EE (B^{d}_{t})^{i_{d}}), 
\end{align*}
where $i_{l}$ is the number of occurences of the letter $e_{l}$ in $w$. It is then clear that $C^{w}(\phi^{1}(t)) = 0$ if any of the $i_{l}$'s is odd. For $w \in \kK_{2n}$, let $2i_{k}$ be the number of occurences of $e_{k}$, then 
\begin{align*}
C^{w}(\phi^{1}(t)) = \frac{1}{(2n)!} (\EE (B^{1}_{t})^{2i_{1}}) \cdots (\EE (B^{d}_{t})^{2i_{d}}), 
\end{align*}
and the conclusion of the lemma follows from the Gaussian moments. 

\end{proof}

\bigskip

\begin{cor} \label{coefficients for square words}
For any $w \in \sS_{2n}$, we have
\begin{align*}
C^{w}(\phi^{M}(T)) \leq C^{w}(\phi(T)). 
\end{align*}
\end{cor}
\begin{proof}
It suffices to show that $C^{w}(\phi^{M}(T)) \leq \frac{1}{n!} \bigg( \frac{T}{2} \bigg)^{n}$. In fact, by independent increments of Brownian motion, we have $\phi^{M}(T) = \phi^{1}(\Delta t)^{\otimes M}$, which implies
\begin{align*}
C^{w}(\phi^{M}(T)) = \sum C^{v_{1}}(\phi^{1}(\Delta t)) \cdots C^{v_{M}}(\phi^{1}(\Delta t)), 
\end{align*}
where $\Delta t = \frac{T}{M}$, and the sum is taken over all $v_{1} * \cdots *v_{M}$ such that each $v_{j}$ is in $\sS_{2k}$ for some $k$. By Lemma \ref{expected signature for piecewise linear approximations}, we have
\begin{align*}
C^{w}(\phi^{M}(T)) &\leq \bigg( \frac{\Delta t}{2} \bigg)^{n} \sum_{k_{1} + \cdots + k_{M} = n} \begin{pmatrix} n \\ k_{1}, \cdots, k_{M} \end{pmatrix} \\
&= \frac{1}{n!} \bigg( \frac{T}{2} \bigg)^{n}, 
\end{align*}
where we have used the fact that $\lambda_{v_{j}} \leq 1$, and each $v_{j}$ has even length. 
\end{proof}

\bigskip

\begin{lem} \label{full ranks in norm}
For each $n, M \in \NN$ and $T \geq 0$, we have
\begin{align*}
\left\| \pi_{2n}(\phi(T)) \right\| = \left\| \pi_{2n}(\phi^{M}(T))  \right\| = \frac{1}{n!} \cdot \bigg( \frac{dT}{2} \bigg)^{n}. 
\end{align*}
\end{lem}
\begin{proof}
That $\left\| \pi_{2n}(\phi(T)) \right\| = \frac{1}{n!} \cdot \bigg( \frac{dT}{2} \bigg)^{n}$ is immediate from Proposition \ref{expected signature for Brownian motion}. In order the prove the second one, we note that
\begin{align} \label{norm of each piece}
\left\| \pi_{2n} (\phi^{1}(t))  \right\| = \frac{1}{n!} \bigg(\frac{dt}{2} \bigg)^{n}
\end{align}
for all $n$ and $t$. By independent increments of Brownian motion, we have
\begin{align*}
\pi_{2n}(\phi^{M}(T))  = \sum_{k_{1} + \cdots + k_{M} = n} \pi_{2k_{1}}(\phi^{1}(\Delta t)) \otimes  \cdots \otimes \pi_{2k_{M}}(\phi^{1}(\Delta t)). 
\end{align*}
By properties of the projective norm and the positivity of all entries, we can change the sum with the norm $\left\|  \cdot  \right\|$, and get
\begin{align*}
\left\| \pi_{2n}(\phi^{M}(T))  \right\| = \sum_{k_{1} + \cdots + k_{M} = n} \left\| \pi_{2k_{1}}(\phi^{1}(\Delta t)) \right\| \cdots \left\| \pi_{2k_{M}}(\phi^{1}(\Delta t)) \right\|. 
\end{align*}
By \eqref{norm of each piece} and the multinomial theorem, we get
\begin{align*}
\left\| \pi_{2n}(\phi^{M}(T)) \right\| = \frac{1}{n!} \bigg( \frac{dT}{2} \bigg)^{n}, 
\end{align*}
thus proving the lemma. 
\end{proof}

\bigskip

Note that the above lemma is true only for projective norm. For Hilbert Schmidt norm, we have $\left\| \pi_{2n}(\phi(T)) \right\| > \left\| \pi_{2n}(\phi^{M}(T))  \right\|$. The next proposition will be very useful for proving the main theorem. It is an immediate consequence of the previous lemma.

\bigskip

\begin{prop} \label{expression from the symmetry}
$\left\| \pi_{2n}(\phi(T)) - \pi_{2n}(\phi^{M}(T)) \right\| = 2 \sum_{w \in \kK_{2n} \setminus \sS_{2n}} C^{w}(\phi^{M}(T))$. 
\end{prop}
\begin{proof}
By Corollary \ref{coefficients for square words}, we have
\begin{align*}
\left\| \pi_{2n}(\phi(T)) - \pi_{2n}(\phi^{M}(T)) \right\| = \sum_{w \in \kK_{2n} \setminus \sS_{2n}} C^{w}(\phi^{M}(T)) + \sum_{w \in \sS_{2n}}[C^{w}(\phi(T)) - C^{w}(\phi^{M}(T))]. 
\end{align*}
Also, Lemma \ref{full ranks in norm} implies that the two terms on the right hand side are equal. Thus, we arrive at the conclusion of the proposition. 
\end{proof}

\bigskip

\section{Proof of Theorem \ref{main theorem}}

This section is devoted to the proof of Theorem \ref{main theorem}. The first part of the theorem is an immediate consequence of Proposition \ref{expected signature for Brownian motion} and Lemma \ref{expected signature for piecewise linear approximations}. To prove the second part, we need a more detailed study of the coefficients of words in $\kK_{2n}$. By Proposition \ref{expression from the symmetry}, it suffices to consider the words in $\kK_{2n} \setminus \sS_{2n}$. Let
\begin{align*}
\eE = \{e_{i}e_{j}e_{i}e_{j}, e_{i}e_{j}e_{j}e_{i}: 1 \leq i, j \leq d, i \neq j\}. 
\end{align*}
For each $k = 0, 1, \cdots, n-2$, define
\begin{align*}
\wW_{2n}^{k} = \{ v * v' * v'': v \in \sS_{2k}, v' \in \eE ,v'' \in \sS_{2n-4-2k} \}, 
\end{align*}
and let 
\begin{align*}
\wW_{2n} := \bigcup_{k=1}^{n-2} \wW_{2n}^{k}. 
\end{align*}
Then $\wW_{2n} \subset \kK_{2n} \setminus \sS_{2n} $. We will show that for each $n$, the set of words whose coeffieicents are of order $\frac{1}{M}$ is precisely $\wW_{2n} \cup \sS_{2n}$. We then compute the sum of coefficients (with absolute values) in $\wW_{2n}$, and those in $\sS_{2n}$ will be obtained by symmetry. We now study the coefficients of words in $\kK_{2n} \setminus (\sS_{2n} \cup \wW_{2n})$ and in $\wW_{2n}$, respectively.

\bigskip

\subsection{Words with negligible coefficients}

The purpose of this part is to show that for each $n$, there exists a constant $C = C(d,n)$ such that
\begin{align} \label{negligible part}
\sum_{w \in \kK_{2n} \setminus (\sS_{2n} \cup \wW_{2n})} C^{w}(\phi^{M}(T)) < \frac{C T^{n}}{M^{2}}
\end{align}
for all large $M$. For $w \in \kK_{2n}$ with $w = e_{i_{1}}e_{i_{2}} \cdots e_{i_{2n-1}}e_{i_{2n}}$, let
\begin{align*}
p(w) = |\{k: i_{2k-1} \neq i_{2k}\}|. 
\end{align*}
In other words, $p(w)$ counts the number of non-square pairs in the word $w$. For each $k = 0, \cdots, n$, define
\begin{align*}
\pP_{2n}^{k} = \{w \in \kK_{2n}: p(w) = k\}. 
\end{align*}
It is clear that $\pP_{2n}^{0} = \sS_{2n}$, $\pP_{2n}^{1}$ is empty, $\wW_{2n} \subset \pP_{2n}^{2}$, and
\begin{align*}
\kK_{2n} = \bigcup_{k=0}^{n} \pP_{2n}^{k}
\end{align*}
as a disjoint union. We will now show that for any $w \in \pP_{2n}^{k}$, we have
\begin{align} \label{bounds for words with different number of non-square pairs}
C^{w}(\phi^{M}(T)) < \frac{CT^{n}}{M^{\lfl (k+1)/2  \rfl}}. 
\end{align}
We first consider the case $k=2$. If $w \in \pP_{2n}^{2}$, then it can be expressed as
\begin{align*}
w = \cdots e_{i}e_{j} \cdots e_{i}e_{j} \cdots, \qquad \text{or} \qquad w = \cdots e_{i}e_{j} \cdots e_{j}e_{i} \cdots, 
\end{align*}
where $i \neq j$, and all other pairs are squares. Without loss of generality, we can assume $w$ has the form
\begin{align*}
w = e_{i_{1}}^{2} \cdots e_{i_{a}}^{2} \underbrace{e_{i}e_{j} * u' * e_{i}e_{j}}_{u} e_{j_{1}}^{2} \cdots e_{j_{b}}^{2}, 
\end{align*}
where $u' \in \sS_{2r}, r \geq 0$, and $a + b + r = n-2$. Let $u = e_{i}e_{j} * u' * e_{i}e_{j}$. Since $\phi^{M}(T) = \phi^{1}(\Delta t) ^{\otimes M}$, we have
\begin{align} \label{expand the coefficient into decompositions}
C^{w}(\phi^{M}(T)) = \sum C^{v_{1}}(\phi^{1}(\Delta t)) \cdots C^{v_{M}}(\phi^{1}(\Delta t)), 
\end{align}
where the sum is taken over the collection of words $(v_{1}, \cdots, v_{M})$ such that (i) $v_{1} * \cdots * v_{M} = w$, and (ii) for each $j$, either $v_{j} \in \sS_{2l}$ for some $l \geq 0$, or $v_{j} = v' * u * v''$, where $v' \in \sS_{2a'}, v'' \in \sS_{2b'}$ for some $a',b' \geq 0$\protect\footnote{Condition (ii) guarantees that every term in the sum is positive. In fact, by Lemma \ref{expected signature for piecewise linear approximations}, if $(v_{1}, \cdots, v_{M})$ satisfies condition (i) but not (ii), then we will have
\begin{align*}
C^{v_{1}}(\phi^{1}(\Delta t)) \cdots C^{v_{M}}(\phi^{1}(\Delta t)) = 0. 
\end{align*}
}. The idea is that the two non-square terms must be grouped together (along with any squares between these two pairs, if they exist) in order for the product on the right hand side of \eqref{expand the coefficient into decompositions} not being zero. This will give at most $n-1$ 'atoms' in the decomposition, and the total number of the elements in the sum will be $\oO(M^{n-1})$. 

Formally, by Lemma \ref{expected signature for piecewise linear approximations}, for each decopomsition $(v_{1}, \cdots, v_{M})$ in the sum, we have
\begin{align} \label{bound for a single decomposition}
C^{v_{1}}(\phi^{1}(\Delta t)) \cdots C^{v_{M}}(\phi^{1}(\Delta t)) < \bigg( \frac{\Delta t}{2}  \bigg)^{a+b+k+2} = \bigg( \frac{\Delta t}{2} \bigg)^{n}, 
\end{align}
and we can bound $C^{w}(\phi^{M}(T))$ by counting the number of elements in the sum on the right hand side of \eqref{expand the coefficient into decompositions}. This is exactly the number of nonnegative integer solutions to
\begin{align*}
x_{1} + \cdots + x_{M} = a + b + 1, 
\end{align*}
which equals
\begin{align*}
\begin{pmatrix} M+a+b \\ M-1 \end{pmatrix} = \begin{pmatrix} M+n-2-r \\ n-1-r \end{pmatrix} < (M+n)^{n-1-r}. 
\end{align*}
Combining the above bound with \eqref{bound for a single decomposition}, we have
\begin{align*}
C^{w}(\phi^{M}(T)) < [(M+n)\Delta t]^{n-1-r} (\Delta t)^{r+1} < \bigg (\frac{T + n \Delta t}{2} \bigg)^{n} \cdot \frac{1}{M^{r+1}}, 
\end{align*}
and this is true for all $w \in \pP_{2n}^{2}$. Now, if $w \in \pP_{2n}^{2} \setminus \wW_{2n}$, then $r \geq 1$, and
\begin{align*}
C^{w}(\phi^{M}(T)) < \frac{CT^{n}}{M^{2}}. 
\end{align*}
The argument for $k \geq 3$ is similar. In order to produce more 'atoms', the best possible choice is to group the consecutive two non-square pairs together, and in the case of odd $k$, one atom should contain three non-square pairs\protect\footnote{For example, the three pairs are $e_{1}e_{2}$, $e_{2}e_{3}$ and $e_{3}e_{1}$. }. Below are two figures for even and odd $k$'s, respectively. 
\begin{align*}
&k \phantom{1} \text{even}: \qquad \cdots \underbrace{e_{i_{1}}e_{i_{2}} \cdots e_{i_{3}}e_{i_{4}}}_{u_{1}} \cdots \cdots  \underbrace{e_{i_{k-3}} e_{i_{k-2}} \cdots e_{i_{k-1}}e_{i_{k}}}_{u_{\frac{k}{2}}} \cdots \\
&k \phantom{1} \text{odd}: \qquad \cdots \underbrace{e_{i_{1}}e_{i_{2}} \cdots e_{i_{3}}e_{i_{4}} \cdots e_{i_{5}}e_{i_{6}}}_{u_{1}} \cdots \cdots \underbrace{e_{i_{k-3}} e_{i_{k-2}} \cdots e_{i_{k-1}}e_{i_{k}}}_{u_{\frac{k-1}{2}}} \cdots 
\end{align*}
As we can see, this will give at most $n - \lfl \frac{k+1}{2} \rfl$ 'atoms' in the decompositions. Thus, by the same computation of the number of elements for such decompositions, we can show that
\begin{align*}
C^{w}(\phi^{M}(T)) < \frac{C T^{n}}{M^{\lfl (k+1)/2 \rfl}}
\end{align*}
for all $w \in \pP_{2n}^{k}$ with $k \geq 3$, where $C$ depends on $n$ only. Since
\begin{align*}
\kK_{2n} \setminus (\sS_{2n} \cup \wW_{2n})^{c} = (\pP_{2n}^{2} \setminus \wW_{2n}) \cup \pP_{2n}^{3} \cup \cdots \cup \pP_{2n}^{n}, 
\end{align*}
and note that the number of elements in $\kK_{2n} \setminus (\sS_{2n} \cup \wW_{2n})^{c}$ depends on $d$ and $n$ only, we conclude \eqref{negligible part} with a constant $C = C(d,n)$.

\bigskip

\subsection{Words in $\wW_{2n}$}

Fix $0 \leq k \leq n-2$ and $w_{k} \in \wW_{2n}^{k}$, then
\begin{align*}
w_{k} = e_{i_{1}}^{2} \cdots e_{i_{k}}^{2} * u * e_{j_{1}}^{2} \cdots e_{j_{n-2-k}}^{2}
\end{align*}
where $u \in \eE$ as defined at the beginning of this section. Similar as before, we have
\begin{align*}
C^{w_{k}}(\phi^{M}(T)) = \sum_{\xX(w_{k})}C^{v_{k}^{1}}(\phi^{1}(\Delta t)) \cdots C^{v_{k}^{M}}(\phi^{1}(\Delta t)), 
\end{align*}
where $\xX(w_{k})$ is the set of words $(v_{k}^{1}, \cdots, v_{k}^{M})$ such that (i) $v_{k}^{1} * \cdots * v_{k}^{M} = w$, and (ii) for each $j$, either $v_{j} \in \sS_{2l}$ for some $l \geq 0$, or $v_{j} = u' * u * u''$, where $u' \in \sS_{2a}, u'' \in \sS_{2b}$ for some $a,b \geq 0$. 

Intuitively, when $M$ is large, most contributions of the sum come from the decompositions with the further restriction that $u$ and each single square are located in different $v_{j}$'s. More precisely, let
\begin{align*}
\xX'(w_{k}) := \big\{ v_{k}^{1} * \cdots *v_{k}^{M} = w: \phantom{1} \text{for each} \phantom{1} j \leq M, \phantom{1} v_{k}^{j} = u \phantom{1} \text{or} \phantom{1} e_{l}^{2} \phantom{1} \text{for some} \phantom{1} l \big\}. 
\end{align*}
Then, $\xX'(w_{k}) \subset \xX(w_{k})$, and 
\begin{align*}
|\xX'(w_{k})| = \begin{pmatrix}   M \\ n-1  \end{pmatrix}. 
\end{align*}
Their difference is
\begin{align*}
|\xX(w_{k}) \setminus \xX'(w_{k})| = \begin{pmatrix} M+n-2  \\ n-1   \end{pmatrix} - \begin{pmatrix}  M \\ n-1  \end{pmatrix} = \oO(M^{n-2}). 
\end{align*}
Also, for each $(v_{k}^{1}, \cdots, v_{k}^{M}) \in \xX(w_{k}) \setminus \xX'(w_{k})$, we have
\begin{align} \label{second negligible set}
C^{v_{k}^{1}}(\phi^{1}(\Delta t)) \cdots C^{v_{k}^{M}}(\phi^{1}(\Delta t)) \leq \bigg( \frac{\Delta t}{2} \bigg)^{n}, 
\end{align}
and thus
\begin{align*}
\sum_{\xX(w_{k}) \setminus \xX'(w_{k})} C^{v_{k}^{1}}(\phi^{1}(\Delta t)) \cdots C^{v_{k}^{M}}(\phi^{1}(\Delta t)) = \oO \bigg( \frac{1}{M^{2}} \bigg). 
\end{align*}
On the other hand, for every $(v_{k}^{1}, \cdots, v_{k}^{M}) \in \xX'(w_{k})$, Lemma \ref{expected signature for piecewise linear approximations} implies that
\begin{align*}
C^{v_{k}^{1}}(\phi^{1}(\Delta t)) \cdots C^{v_{k}^{M}}(\phi^{1}(\Delta t)) = \frac{1}{6} \bigg( \frac{\Delta t}{2} \bigg)^{n}. 
\end{align*}
Since $|\xX'(w_{k})| = \begin{pmatrix}   M \\ n-1  \end{pmatrix}$, combining the above equality with \eqref{second negligible set}, we get
\begin{align*}
C^{w_{k}}(\phi^{M}(T)) = \frac{1}{12 \cdot (n-1)!} \bigg( \frac{T}{2} \bigg)^{n-1} \Delta t + \oO \big( \frac{1}{M^{2}} \big), 
\end{align*}
which holds for each $w_{k} \in \wW_{2n}^{k}$. Note that there are $4 d^{n-2} \begin{pmatrix} d \\ 2 \end{pmatrix}$ in $\wW_{2n}^{k}$ for each $k$, summing over $k$ from $0$ to $n-2$, we get
\begin{align} \label{the main contribution}
\sum_{w \in \wW_{2n}} C^{w}(\phi^{M}(T)) = \frac{(d-1)T}{6M \cdot (n-2)!} \bigg( \frac{dT}{2} \bigg)^{n-1} + \oO \big( \frac{1}{M^{2}} \big). 
\end{align}

\bigskip

\subsection{Putting all together}

We are now in a position to prove the main claim. By Proposition \ref{expression from the symmetry}, we have
\begin{align} \label{symmetry}
\left\| \pi_{2n}(\phi^{M}(T)) - \pi_{2n}(\phi(T)) \right\| = 2 \sum_{w \in \kK_{2n} \setminus \sS_{2n}} C^{w}(\phi^{M}(T)
\end{align}
Also by \eqref{negligible part}, we know that the coefficients of the words in $\kK_{2n} \setminus (\sS_{2n} \cup \wW_{2n})$ are of order $O(\frac{1}{M^{2}})$, and thus
\begin{align*}
\sum_{w \in \kK_{2n} \setminus \sS_{2n}} C^{w}(\phi^{M}(T)) = \sum_{w \in \wW_{2n}} C^{w}(\phi^{M}(T)) + \oO \big( \frac{1}{M^{2}} \big). 
\end{align*}
Substituting \eqref{the main contribution} into the right hand side, and combining it with \eqref{symmetry}, we get
\begin{align*}
\left\| \pi_{2n}(\phi^{M}(T)) - \pi_{2n}(\phi(T)) \right\| = \frac{(d-1)T}{3M \cdot (n-2)!} \bigg( \frac{dT}{2} \bigg)^{n-1} + \oO \big(\frac{1}{M^{2}} \big). 
\end{align*}
Multiplying $\frac{M}{T}$ on both sides, and letting $M \rightarrow +\infty$, we get
\begin{align*}
\lim_{M \rightarrow +\infty} \frac{M}{T} \left\| \pi_{2n}(\phi(T)) - \pi_{2n}(\phi^{M}(T)) \right\| = \frac{d-1}{3 \cdot (n-2)!}  \bigg(\frac{dT}{2} \bigg)^{n-1}. 
\end{align*}
Thus we have completed the proof of the main theorem.

\bigskip

\begin{flushleft}
\textsc{Mathematical and Oxford-Man Institutes, University of Oxford, 24-29 St.Giles, Oxford, OX1 3LB, UK.} \\
Email addresses: ni@maths.ox.ac.uk, xu@maths.ox.ac.uk
\end{flushleft}

\end{document}